\documentclass[reqno, a4paper]{amsart}
\usepackage[hmargin=3cm]{geometry}

\usepackage{amsmath}
\usepackage{amsfonts}

\renewcommand{\c}[1]{\overline{#1}}

\newcommand{\D}{\mathbb{D}}
\newcommand{\T}{\mathbb{T}}
\newcommand{\C}{\mathbb{C}}

\renewcommand{\Re}{\mathrm{Re}\,}

\newcommand{\upto}{\nearrow}

\renewcommand{\epsilon}{\varepsilon}
\renewcommand{\phi}{\varphi}

\newcommand{\Lloc}{L_{loc}^1}

\newcommand{\dd}{\partial}
\newcommand{\dbar}{\c{\dd}}

\renewcommand{\and}{\quad\text{and}\quad}

\newcommand{\ball}[2]{D\left(#1,#2\right)}
\newcommand{\action}[2]{\left\langle #1, #2 \right\rangle}
\DeclareMathOperator{\supp}{supp}

\let\originalleft\left
\let\originalright\right
\renewcommand{\left}{\mathopen{}\mathclose\bgroup\originalleft}
\renewcommand{\right}{\aftergroup\egroup\originalright}

\newtheorem{theorem}{Theorem}

\newtheorem{lemma}{Lemma}
\newtheorem{proposition}{Proposition}

\theoremstyle{definition}
\newtheorem*{definition}{Definition}
\newtheorem*{remark}{Remark}

\title[Solving Poisson's equation for the standard weighted Laplacian]{Solving Poisson's equation for the standard weighted Laplacian in the unit disc}
\date{\today}
\author{Gustav Behm}
\address{Department of Mathematics, KTH Royal Institute of Technology, 100 44 Stockholm, Sweden}
\email{gbehm@math.kth.se}



\begin{document}

\begin{abstract}
We will find Green's function for the standard weighted Laplacian and use the corresponding Green's potential to solve Poisson's equation in the unit disc  with zero boundary values, in the sense of radial $L^1$-means, for complex Borel measures $\mu$ satisfying the condition
$$\int_\D (1-|w|^2)^{\alpha+1} d|\mu|(w) < \infty, \quad \alpha >-1.$$
\end{abstract}

\maketitle

\section{Introduction}
The standard weighted (negative) Laplacian we are going to study is defined as
$$L_\alpha=-\dbar_z (1-|z|^2)^{-\alpha}\dd_z, \quad \alpha>-1.$$
Here $\dd_z$ and $\dbar_z$ denote the two Wirtinger derivatives:
$$\dd_z = \frac{1}{2}\left(\frac{\dd}{\dd x}-i\frac{\dd}{\dd y}\right) \quad\text{and}\quad \dbar_z = \frac{1}{2}\left(\frac{\dd}{\dd x}+i\frac{\dd}{\dd y}\right), \quad z=x+iy\in \C,$$
where we understand the derivatives in the distributional sense when necessary.
The operator $\dbar_z$ is sometimes referred to as the Cauchy-Riemann operator since the equation $\dbar_z u = 0$ is equivalent to the Cauchy-Riemann equation.

This type of weighted Laplacian $\dbar_z \rho^{-1} \dd_z$ was studied for weights $\rho$ which are continuously differentiable in a neighborhood of its domain by Garabedian in his paper \cite{garabedian}.

In our case the weight will not be continuously differentiable on the boundary of the unit disc and we need to be more careful close to the boundary.
This weight occurs often in the study of weighted Bergman spaces in the unit disc, see \cite{h-book}.

We will consider our weighted Laplacian as an operator on the space of distributions on $\D$.
If we define the adjoint of the operator $L_\alpha$ as $\c{L_\alpha}=-\dd_z (1-|z|^2)^{-\alpha}\dbar_z$ we get explicitly
\begin{equation}
\label{eq:L-def}
\action{L_\alpha u}{\phi}=\action{u}{\c{L_\alpha} \phi}, \quad \phi\in C_0^\infty(\D),
\end{equation}
since $(1-|z|^2)^{-\alpha}$ is smooth inside $\D$ so that the multiplication is well-defined.

The aim of this paper is to solve Poisson's equation for $L_\alpha$ in the unit disc $\D$:
\begin{enumerate}
\item \text{$L_\alpha u = \mu$ in sense of distributions on $\D$ and}
\item \text{$u(re^{i\cdot}) \to 0$ in $L^1(\T)$ when $r\upto 1$}
\end{enumerate}
for any complex Borel measure $\mu$ satisfying
$$\int_\D (1-|w|^2)^{\alpha+1} d|\mu|(w) < \infty.$$


Our approach to solving Poisson's equation is to find Green's function for $L_\alpha$ and verifying that the corresponding Green's potential solves the equation.
The uniqueness of the solution is then implied by the uniqueness of the Dirichlet problem for the same operator which has been studied in \cite{ow} by Olofsson and Wittsten where they found the Poisson kernel for $L_\alpha$ in the unit disc.

\section{Notation and prerequisites}

We adopt the convention that the constant $C$ appearing in the equations below is not necessarily the same at every instance and depends only on its indices, e.g.~$C_\gamma$ depends only on $\gamma$.

The parameter $\alpha$ will appear frequently and we will throughout this text assume that $\alpha>-1$, even when this is not stated explicitly.

Let $\ball{z}{R}$ be the open disc in the complex plane centered at $z$ with radius $R$.
Let $\D=\ball{0}{1}$ be the unit disc and $\T=\dd\D$ the unit circle.
The area measure $dA$ will be normalized as $dA=\frac{1}{\pi}dxdy$ and the arc length measure $ds$ along $r(t)$ defined and normalized as $ds=\frac{1}{2\pi}|r'(t)|dt$.

In this setting the classical or unweighted case $\alpha=0$ becomes $L_0=-\dbar\dd=-\frac{1}{4}\Delta$ and the corresponding classical Green's function is
$$G(z,w)=-\log\left|\frac{z-w}{1-\c{z}w}\right|^2.$$

Let $\delta_0$ denote the Dirac delta distribution with support only at $0$.
Furthermore we will consider the expression $\delta_0(z-\zeta)$ as the Dirac delta distribution with respect to $z$ with support only at $\zeta$, and similarly if we interchange the roles of $z$ and $\zeta$.

We will abuse notation and identify a Radon measure $\mu$ on $\D$ with the distribution on $\D$ with the same name by its usual duality property:
\begin{equation}
\label{eq:radon-action}
\action{\mu}{\phi} = \int_\D \phi d\mu, \quad \phi\in C_0^\infty(\D)
\end{equation}
Note that any locally finite Borel measure on $\D$ is a Radon measure (see \cite{rudin} Theorem 2.17 for inner regularity).

Similarly we will also identify functions $u\in\Lloc(\D)$ with distributions on $\D$ with the same name having the action
$$\action{u}{\phi}=\int_\D u \phi dA, \quad \phi\in C_0^\infty(\D).$$
Hence we can speak of the Wirtinger derivatives of functions in $u\in\Lloc(\D)$.
Explicitly we mean that the distribution $\dd_z u$ have the action
$$\action{\dd_z u}{\phi}=-\int_\D u \dd_z\phi dA, \quad \phi\in C_0^\infty(\D)$$
and similarly for the $\dbar_z$-derivative.

\section{Green's function}

We seek Green's function for $L_\alpha=-\dbar_z|z|^{-2\alpha}\dd_z$ in $\D$, meaning that we seek the unique function $G_\alpha(z,w)$ satisfying the following definition.
\begin{definition}
Green's function $G_\alpha(z,w)$ is the fundamental solution of $L_\alpha$ with zero boundary values.
That is, $G_\rho(z,w)$ is the function defined on $\D\times\D$ solving for fixed $w\in\D$:
\begin{enumerate}
\item $L_\alpha G_\alpha(z,w)=\delta_0(z-w)$,
\item $G_\alpha(z,w)\to 0$ when $z\to\zeta\in\dd\D$,
\item $G_\alpha(z,w)=\c{G_\alpha(w,z)}$.
\end{enumerate}
\end{definition}

We will not explicitly prove the the symmetry condition $G_\alpha(z,w)=\c{G_\alpha(w,z)}$ as it will be quite clear that it is satisfied just by inspection of the soon to be proposed expression.

The representation for Green's function and its proof involves complex exponentiation of the type $z^\gamma$ for real $\gamma$ where $z$ will lie in the right half-plane $\Re z > 0$.
Hence we restrict ourselves to the principal branch, that is, the branch for which $z^\gamma$ is real for real $z$.

\begin{theorem}
For the principal branch of the complex exponential Green's function $G_\alpha(z,w)$ for the operator $L_\alpha=-\dbar_z (1-|z|^2)^{-\alpha}\dd_z$ for $\alpha>-1$ in $\D$ is given by
$$G_\alpha(z,w)=(1-\c{z}w)^\alpha h\circ g(z,w), \quad z\neq w,$$
where
$$h(s)=\int_0^s\frac{t^\alpha}{1-t} dt=\sum_{n=0}^\infty \frac{s^{\alpha+1+n}}{\alpha+1+n}, \quad 0\leq s < 1,$$
and
$$g(z,w)=1-\left|\frac{z-w}{1-\c{z}w}\right|^2=\frac{(1-|z|^2)(1-|w|^2)}{|1-z\c{w}|^2}.$$
\end{theorem}

Observe that since $|\c{z}w|<|w|<1$ for $z,w\in\D$, the expression $1-\c{z}w$ lies strictly in the right half-plane and is bounded away from both zero and infinity.

Since we already have uniqueness of Dirichlet's problem (see \cite{ow}) the theorem is proved by verifying the definition of Green's function.
The boundary condition follows quite easily from the estimates in Section~\ref{s:estimates} below and the condition $L_\alpha G_\alpha(z,w)=\delta_0(z-w)$ is verified in Proposition~\ref{p:LG} in Section~\ref{s:LG}

\begin{remark}
The proposed expression for Green's function can be expressed using the incomplete Beta function $B(x;a,b)=\int_0^x t^{a-1}(1-t)^{b-1} dt$ as
$$G_\alpha(z,w)=(1-\c{z}w)^\alpha B\left(\frac{(1-|z|^2)(1-|w|^2)}{|1-z\c{w}|^2}; \alpha+1,0\right)$$
and therefore also using the zero-balanced Gauss' hypergeometric function ${}_2F_1(1,\alpha+1;\alpha+2;z)$ as
$$G_\alpha(z,w)=\frac{1}{\alpha+1}\frac{(1-|z|^2)^{\alpha+1}(1-|w|^2)^{\alpha+1}}{(1-\c{z}w)(1-z\c{w})^{\alpha+1}} {}_2F_1\left(1,\alpha+1;\alpha+2;\frac{(1-|z|^2)(1-|w|^2)}{|1-z\c{w}|^2}\right).$$
For definitions and formulas for $B$ and ${}_2F_1$, as well as their relation to each other, we refer the reader to \cite{ab}.
\end{remark}

\subsection{A closer look at $h(s)$ and $g(z,w)$}
\label{s:estimates}
We start with looking at the behavior of $g(z,w)$.
If we fix $w\in\D$ then we can consider the M\"{o}bius transformation
$$\phi_w(z)=\frac{z-w}{1-z\c{w}}$$
which maps $\c{\D}$ onto $\c{\D}$ and in particular maps $\T$ onto $\T$.
Then
\begin{equation}
\label{eq:g-phi}
g(z,w)=1-|\phi_w(z)|^2.
\end{equation}
Using this we can show the following properties of $g(z,w)$:
\begin{lemma}
The function $g(z,w)$ satisfies:
\begin{enumerate}
\item $0\leq g(z,w) \leq 1$,
\item $g(z,w)=1$ if and only if $z=w$ and
\item for $w\in\D$: $g(z,w)=0$ if and only if $z\in\T$.
\end{enumerate}
\end{lemma}

\begin{proof}
Since $\phi_w$ maps $\c{\D}$ into $\c{\D}$ we have $0\leq |\phi_w(z)| \leq 1$.
Now $g(z,w)=1$ if and only if $\phi_w(z)=0$, but since $\phi_w$ is a bijection and maps $w$ to the origin then $\phi_w(z)=0$ if and only if $z=w$.
If we reason similarly for $z\in\T$ we can deduce that $g(z,w)$ is zero if and only if $z\in\T$.
\end{proof}

Turning our attention to $h(s)$ note that it is chosen to satisfy
\begin{equation}
\label{eq:h-derivative}
h'(s)=\frac{s^\alpha}{1-s}.
\end{equation}
Hence we can think of $h$ as a generalized logarithm in some sense and in particular we have $h(s)=-\log(1-s)$ when $\alpha=0$.

\begin{lemma}
\label{l:h-est}
The function $h(s)$ satisfies the estimate
$$h(s) \leq C_\alpha s^{\alpha+1}\left(1 - \log(1-s)\right)$$
where $C_\alpha$ is a constant depending only on $\alpha>-1$.
\end{lemma}

\begin{proof}
If we isolate the first term in the series representation and estimate using the fact that $\alpha+1>0$ we get
\begin{equation*}
h(s)=\frac{s^{\alpha+1}}{\alpha+1} + \sum_{n=1}^\infty \frac{s^{\alpha+1+n}}{\alpha+1+n} \leq \frac{s^{\alpha+1}}{\alpha+1} + s^{\alpha+1} \sum_{n=1}^\infty \frac{s^n}{n} =  \frac{s^{\alpha+1}}{\alpha+1} - s^{\alpha+1}\log(1-s).
\end{equation*}
If we choose $C_\alpha$ appropriately we get the desired estimate.
\end{proof}

From this estimate we can derive the following fundamental estimate for $G_\alpha(z,w)$.

\begin{lemma}
\label{l:Gest}
The function $G_\alpha(z,w)$ satisfy the estimate
$$|G_\alpha(z,w)| \leq C_\alpha \left( \frac{(1-|w|^2)^{\alpha+1}(1-|z|^2)^{\alpha+1}}{|1-\c{z}w|^{\alpha+2}} + (1-|w|^2)^\alpha G(z,w)\right).$$
\end{lemma}

\begin{proof}
If we apply the previous lemma to $h\circ g(z,w)$ we get
\begin{align}
h\circ g(z,w) & \leq C_\alpha g(z,w)^{\alpha+1}(1 - \log(1-g(z,w))) = C_\alpha g(z,w)^{\alpha+1}\left(1 - \log\left|\frac{z-w}{1-\c{z}w}\right|^2 \right) \nonumber\\
& = C_\alpha g(z,w)^{\alpha+1}(1 + G(z,w)) \label{eq:hg-est}.
\end{align}
Using this estimate we get:
$$|G_\alpha(z,w)| \leq C_\alpha|1-\c{z}w|^\alpha g(z,w)^{\alpha+1} (1 + G(z,w))$$
where explicitly
$$|1-\c{z}w|^\alpha g(z,w)^{\alpha+1} = \frac{(1-|z|^2)^{\alpha+1}(1-|w|^2)^{\alpha+1}}{|1-\c{z}w|^{\alpha+2}}.$$
In particular we can factor out $(1-|w|^2)^\alpha$ so that the remaining factor can be written as
$$\frac{(1-|z|^2)^{\alpha+1}(1-|w|^2)}{|1-\c{z}w|^{\alpha+2}} = \frac{(1-|z|^2)^{\alpha+1}}{|1-\c{z}w|^{\alpha+1}} \frac{1-|w|^2}{|1-\c{z}w|}.$$
This expression is bounded since both the factors $\frac{1-|z|^2}{|1-\c{z}w|}$ and $\frac{1-|w|^2}{|1-\c{z}w|}$ can be bounded by using the reverse triangle inequality:
\begin{equation}
\label{eq:reverse}
\frac{1-|z|^2}{|1-\c{z}w|}\leq \frac{1-|z|^2}{1-|z||w|} \leq \frac{(1-|z|)(1+|z|)}{1-|z|} = 1 + |z| < 2
\end{equation}
and similarly for $\frac{1-|w|^2}{|1-\c{z}w|}$.
So if we move $|1-\c{z}w|^\alpha g(z,w)^{\alpha+1}$ inside the parenthesis in (\ref{eq:hg-est}) and leave it untouched for the first term and use the just derived bound for the second term we get the desired inequality for $G_\alpha(z,w)$.
\end{proof}

\begin{remark}
That $G_\alpha(z,w)$ should satisfy an estimate of this type can be seen hinted at in Garabedian's paper \cite{garabedian}, where he quantifies Green's function for the weighted Laplacian with weight $\rho\in C^1(\c{\D})$ as:
$$G_\rho(z,w)=-\rho(w)\log|z-w| + \text{continuous terms}.$$
\end{remark}

These estimates allows us to quickly get the desired boundary behavior of $G_\alpha(z,w)$.
\begin{proposition}
For fixed $w\in\D$ then $G_\alpha(z,w)\to 0$ when $|z|\upto 1$.
\end{proposition}

\begin{proof}
We will use the estimate in the previous lemma and show that both terms have the desired limit.
If we fix $w\in\D$ then $G(z,w)\to 0$ as $|z|\to 1$ so the second term is immediately clear.
The first term we estimate using the triangle inequality as
$$\frac{(1-|w|^2)^{\alpha+1}(1-|z|^2)^{\alpha+1}}{|1-\c{z}w|^{\alpha+2}} \leq (1-|z|^2)^{\alpha+1} \frac{(1-|w|^2)^{\alpha+1}}{(1-|w|)^{\alpha+2}}$$
which tends to zero as $|z|\to 1$ for fixed $w\in\D$.
\end{proof}

Having established the boundary condition we turn our attention towards showing the remaining condition $L_\alpha G_\alpha(z,w)=\delta_0(z-w)$ and to do this we use the same estimate to show the following lemma:

\begin{lemma}
\label{l:Gint}
The function $G_\alpha(\cdot,w)$ belongs to $\Lloc(\D)$ for fixed $w\in\D$.
\end{lemma}

\begin{proof}
From the proof of Lemma~\ref{l:Gest} we see that we can estimate $G_\alpha(z,w)$ as
\begin{equation}
\label{eq:Gest2}
|G_\alpha(z,w)| \leq C_\alpha (1-|w|^2)^{\alpha}  \left(1 + G(z,w)\right)
\end{equation}
if we just factor out $(1-|w|^2)^\alpha$ from the first term and use the estimates of the remaining factors mentioned above.
From this the lemma follows because $G(z,w)\in\Lloc(\D)$ for fixed $w\in\D$.
\end{proof}

\subsection{Verifying $L_\alpha G_\alpha(z,w)=\delta_0(z-w)$}
\label{s:LG}
Taking the $\dd_z$-derivative of (\ref{eq:g-phi}) we get after some calculations that
$$\dd_z g(z,w) = - \dd_z(\phi_w(z)\c{\phi_w(z)}) = \frac{1-|w|^2}{(1-z\c{w})} \frac{1}{w-z} |\phi_w(z)|^2, \quad z\neq w.$$
So for $z\neq w$ we can consider the $\dd_z$-derivative of the composition:
$$\dd_z h\circ g(z,w) = \frac{g^\alpha}{1-g} \dd_z g(z,w) = \frac{g^\alpha}{|\phi_w(z)|^2} \dd_z g(z,w) = g^\alpha \frac{1-|w|^2}{(1-z\c{w})} \frac{1}{w-z}.$$
After inserting the expression of $g$ and some simplifications we get:
$$\dd_z h\circ g(z,w)=\frac{(1-|z|^2)^\alpha}{(1-\c{z}w)^\alpha} \frac{(1-|w|^2)^{\alpha+1}}{(1-z\c{w})^{\alpha+1}} \frac{1}{w-z}, \quad z\neq w.$$
Therefore we define
\begin{equation}
\label{eq:F}
F_{\alpha,w}(z)=(1-\c{z}w)^\alpha \dd_z h\circ g(z,w)=\frac{(1-|w|^2)^{\alpha+1}}{(1-z\c{w})^{\alpha+1}} \frac{(1-|z|^2)^\alpha}{w-z}, \quad z\neq w.
\end{equation}
If we observe that the factor $(1-\c{z}w)^\alpha$ is anti-analytic in $\D$ and hence vanish under the $\dd_z$-derivative we get
\begin{equation}
\label{eq:dzG}
\dd_z G_\alpha(z,w) = F_{\alpha,w}(z), \quad z\neq w.
\end{equation}
To show that $F_{\alpha,w}$ wholly determines the weak $\dd_z$-derivative of $G_\alpha$ we proceed by noting that $F_{\alpha,w}\in\Lloc(\D)$ and so defines a distribution in the desired sense.
Indeed, if we estimate $\displaystyle \frac{1-|w|^2}{1-z\c{w}}$ as we did in (\ref{eq:reverse}) we see that the first factor in $F_{\alpha,w}$ is bounded and the second factor $\displaystyle \frac{(1-|z|^2)^\alpha}{w-z}$ belongs to $\Lloc(\D)$.

Now we can calculate the weak $\dd_z$-derivative of $G_\alpha(z,w)$:
\begin{lemma}
For fixed $w\in\D$ we have
$$\dd_z G_\alpha(z,w)=F_{\alpha,w}(z)=\frac{(1-|w|^2)^{\alpha+1}}{(1-z\c{w})^{\alpha+1}} \frac{(1-|z|^2)^\alpha}{w-z}$$
in the sense of distributions on $\D$.
\end{lemma}

\begin{proof}
Lemma~\ref{l:Gint} says that $G_\alpha(\cdot,w)\in\Lloc(\D)$ so that its derivative has the action
$$\action{\dd_z G_\alpha(z,w)}{\phi(z)}=-\int_\D G_\alpha(z,w)\dd_z\phi dA(z), \quad \phi\in C_0^\infty(\D).$$
By Lebesgue's dominated convergence theorem we can find the action by the limit
\begin{equation}
\label{eq:intlim}
\int_\D G_\alpha(z,w)\dd_z\phi dA(z)=\lim_{\epsilon\to 0} \int_{\D\setminus \ball{w}{\epsilon}} G_\alpha(z,w)\dd_z\phi dA(z).
\end{equation}
For any small $\epsilon>0$ let $D_\epsilon=\ball{w}{\epsilon}$ so that partial integration (remembering the normalizations of the measures) gives
$$\int_{\D\setminus D_\epsilon} G_\alpha(z,w)\dd_z\phi dA(z) = \int_{\dd D_\epsilon} G_\alpha(z,w) \phi \c{\nu} ds(z) - \int_{\D\setminus D_\epsilon}(\dd_z G_\alpha(z,w))\phi dA(z),$$
where $\nu$ is the unit outward normal of $\D\setminus D_\epsilon$, that is, the inward unit normal of $D_\epsilon$.

We want to show that the line integral vanish and therefore estimate it as
$$\left|\int_{\dd D_\epsilon} G_\alpha(z,w) \phi \c{\nu} ds(z)\right| \leq C_\phi \int_{\dd D_\epsilon} |G_\alpha(z,w)| ds(z)$$
where $C_\phi$ is a constant depending only on $\sup\phi$.
Here we go back to the estimate (\ref{eq:Gest2}) in Lemma~\ref{l:Gint} and we get:
$$|G_\alpha(z,w)| \leq C_{\alpha,w}\left(1 - \log\left|\frac{z-w}{1-\c{z}w}\right|^2\right)$$
so if we consider this for $z$ on $\dd D_\epsilon$, and change the constant appropriately, we see that
$$|G_\alpha(z,w)| \leq C_{\alpha,w}(1 - \log \epsilon^2)$$
since $|z-w|=\epsilon$ and $|1-\c{z}w|$ is bounded away from zero.
So we get the limit
\begin{align*}
\left|\int_{\dd D_\epsilon} G_\alpha(z,w) \phi \c{\nu} ds(z)\right| & \leq C_{\phi,\alpha,w} (1-\log \epsilon^2) \int_{\dd D_\epsilon} ds = C_{\phi,\alpha,w} (1-\log \epsilon^2) \epsilon \to 0 \quad \text{as $\epsilon\to 0$.}
\end{align*}

From (\ref{eq:dzG}) we have $\dd_z G_\alpha(z,w)=F_{\alpha,w}(z)$ when $z\neq w$ and therefore
\begin{equation*}
\action{\dd_z G_\alpha(z,w)}{\phi(z)}=\lim_{\epsilon\to 0} \int_{\D\setminus D_\epsilon} (\dd_z G_\alpha(z,w))\phi dA(z)=\int_\D F_{\alpha,w}(z) \phi dA(z) = \action{F_{\alpha,w}}{\phi}. \qedhere
\end{equation*}
\end{proof}

Given this lemma we can show the last remaining condition for Green's function.

\begin{proposition}
\label{p:LG}
For fixed $w\in\D$ it holds that
$$L_\alpha G_\alpha(z,w)= \delta_0(z-w)$$
in the sense of distributions on $\D$.
\end{proposition}

\begin{proof}
By the definition $L_\alpha=-\dbar_z (1-|z|^2)^{-\alpha} \dd_z$ and the previous lemma we get
$$\action{L_\alpha G_\alpha(z,w)}{\phi(z)} = \action{\dd_z G_\alpha(z,w)}{(1-|z|^2)^{-\alpha}\dbar_z\phi(z)} = \action{F_{\alpha,w}(z)}{(1-|z|^2)^{-\alpha}\dbar_z\phi(z)}.$$
Inserting the expression for $F_{\alpha,w}$ we get cancellation of the two smooth factors $(1-|z|^2)^\alpha$ and $(1-|z|^2)^{-\alpha}$:
\begin{align*}
\action{L_\alpha G_\alpha(z,w)}{\phi(z)} & = \action{\frac{(1-|w|^2)^{\alpha+1}}{(1-z\c{w})^{\alpha+1}} \frac{(1-|z|^2)^\alpha}{w-z}}{(1-|z|^2)^{-\alpha}\dbar_z\phi(z)}\\
&  = \action{\frac{(1-|w|^2)^{\alpha+1}}{(1-z\c{w})^{\alpha+1}} \frac{1}{w-z}}{\dbar_z\phi(z)}.
\end{align*}
Here observe that the first factor $\frac{(1-|w|^2)^{\alpha+1}}{(1-z\c{w})^{\alpha+1}}$ is analytic in $\D$ and therefore
\begin{align*}
\action{\frac{(1-|w|^2)^{\alpha+1}}{(1-z\c{w})^{\alpha+1}} \frac{1}{w-z}}{\dbar_z\phi(z)} & =\action{\frac{1}{w-z}}{ \dbar_z\left(\frac{(1-|w|^2)^{\alpha+1}}{(1-z\c{w})^{\alpha+1}} \phi(z)\right)}\\
& =\action{ \dbar_z\frac{1}{z-w}}{\frac{(1-|w|^2)^{\alpha+1}}{(1-z\c{w})^{\alpha+1}} \phi(z)}.
\end{align*}
It is well-know that $\dbar_z\frac{1}{z-w}=\delta_0(z-w)$ (the uninitiated reader may see Lemma 20.3 in \cite{rudin} or \cite{h}) so
$$\action{ \dbar_z\frac{1}{z-w}}{\frac{(1-|w|^2)^{\alpha+1}}{(1-z\c{w})^{\alpha+1}} \phi(z)} = \frac{(1-|w|^2)^{\alpha+1}}{(1-|w|^2)^{\alpha+1}} \phi(w) = \phi(w),$$
where we use the assumption that we have the principal branch of the complex exponentiation.
Hence the desired identity holds.
\end{proof}

The last two propositions verified the definition of Green's function and therefore the theorem is proved by referring to the uniqueness of the Dirichlet problem.
Now we use Green's function to solve Poisson's equation.

\section{Poisson's equation}
We will construct or solution using the following potential:

\begin{definition}
Given a complex Borel measure $\mu$ on $\D$ which satisfies
\begin{equation}
\label{eq:measure-condition}
\int_\D (1-|w|^2)^{\alpha+1} d|\mu|(w) < \infty
\end{equation}
define \textit{Green's potential} as
$$G_\alpha^\mu(z) = \int_\D G_\alpha(z,w) d\mu(w)$$
where $G_\alpha(z,w)$ is Green's function for $L_\alpha$.
\end{definition}

Note that any such measure $\mu$ will be locally finite and therefore a Radon measure.
Indeed, for any compact $K\subset\D$ let $M=\supp_K (1-|w|^2)^{-(\alpha+1)}$ such that
$$|\mu(K)| \leq \int_K d|\mu| \leq M\int_K (1-|w|^2)^{\alpha+1} d|\mu|(w) \leq M\int_\D (1-|w|^2)^{\alpha+1} d|\mu|(w) < \infty.$$
Hence we can consider $\mu$ as a distribution on $\D$ with the action defined in (\ref{eq:radon-action}).

\begin{theorem}
\label{t:poisson}
The Green potential $G_\alpha^\mu(z)$ for $\mu$ satisfying (\ref{eq:measure-condition}) is the unique solution to Poisson's equation for $L_\alpha$ ($\alpha>-1$):
\begin{enumerate}
\item \text{$L_\alpha G_\alpha^\mu = \mu$ in sense of distributions on $\D$ and}
\item \text{$G_\alpha^\mu(re^{i\cdot}) \to 0$ in $L^1(\T)$ when $r\upto 1$.}
\end{enumerate}
\end{theorem}

Again the uniqueness is a consequence of the uniqueness of Dirichlet's problem.
The proof of the theorem will be divided into two sections dealing with the two conditions separately.

\subsection{The $L^1(\T)$-boundary limit of $G_\alpha^\mu$}
The condition that $G_\alpha^\mu(re^{i\cdot}) \to 0$ in $L^1(\T)$ when $r\upto 1$ means explicitly that we want the radial $L^1(\T)$-means to tend to zero:
$$\|G_\alpha^\mu(re^{i\cdot})\|_{L^1(\T)} = \int_0^{2\pi} |G_\alpha^\mu(re^{i\theta})| \frac{d\theta}{2\pi} \to 0 \quad\text{when}\quad r\upto 1.$$
If we use the triangle inequality and Fubini's theorem we can study this as:
\begin{align*}
\int_0^{2\pi} |G_\alpha^\mu(re^{i\theta})| \frac{d\theta}{2\pi} & \leq \int_0^{2\pi} \left( \int_\D |G_\alpha(re^{i\theta},w)| d|\mu|(w) \right)  \frac{d\theta}{2\pi} = \int_\D  \left( \int_0^{2\pi} |G_\alpha(re^{i\theta},w)| \frac{d\theta}{2\pi} \right)d|\mu|(w).
\end{align*}

So if we can deduce that we can apply the dominated convergence theorem to the inner integral, considered as a function of $w$ and depending on $r$, and that it tends to zero pointwise with respect to $w$ then the desired boundary limit will hold.
The following proposition will provide the required details.

\begin{proposition}
The function
$$I_r(w)=\int_0^{2\pi} |G_\alpha(re^{i\theta},w)| \frac{d\theta}{2\pi}$$
satisfies the estimate
$$I_r(w) \leq C_{\alpha,r_0} (1-|w|^2)^{\alpha+1}, \quad 0 < r_0 \leq r < 1,$$
and the pointwise limit
$$\text{$I_r(w)\to 0$ when $r\upto 1$ for fixed $w\in\D$.}$$
\end{proposition}

The estimate guarantees that $I_r(w)$ is dominated for $r$ close to 1 by a $\mu$-integrable function.
So that the pointwise limit together with the dominated convergence theorem will indeed give us the desired limit.

We begin by applying Lemma~\ref{l:Gest} and splitting the integral into two integrals, $I_r^1$ and $I_r^2$:
\begin{align}
I_r(w) & = \int_0^{2\pi} |G_\alpha(re^{i\theta},w)| \frac{d\theta}{2\pi} \nonumber\\
& \leq C_\alpha\left( (1-|w|^2)^{\alpha+1} \int_0^{2\pi} \frac{(1-r^2)^{\alpha+1}}{|1-re^{-i\theta}w|^{\alpha+2}} \frac{d\theta}{2\pi} + (1-|w|^2)^\alpha \int_0^{2\pi} G(re^{i\theta},w) \frac{d\theta}{2\pi} \right)\nonumber\\
& = C_\alpha\left( (1-|w|^2)^{\alpha+1} I_r^1(w) + (1-|w|^2)^\alpha I_r^2(w) \right)\label{eq:L1-radial-est}
\end{align}
We will study these two terms separately in the following two lemmas.
The first integral will be found to be bounded by a constant and the second will be dominated by $(1-|w|^2)$ and both will tend to zero pointwise with respect to $w$.

\begin{lemma}
\label{l:Iwr1}
The function
$$I_r^1(w) = \int_0^{2\pi} \frac{(1-r^2)^{\alpha+1}}{|1-re^{-i\theta}w|^{\alpha+2}} \frac{d\theta}{2\pi}$$
is bounded by $1$ and satisfies the pointwise limit
$$\text{$I_r^1(w)\to 0$ when $r\upto 1$ for fixed $w\in\D$.}$$
\end{lemma}

\begin{proof}
The pointwise limit follow from the inverse triangle inequality applied to the integrand:
$$\frac{(1-r^2)^{\alpha+1}}{|1-re^{-i\theta}w|^{\alpha+2}} \leq \frac{(1-r^2)^{\alpha+1}}{(1-|w|)^{\alpha+2}}$$
which tends to zero when $r\upto 1$ and $w$ is fixed in $\D$.
Integrating this inequality gives the desired pointwise limit.

To show the boundedness we will control the integrand with $L^1$-means of a kernel of the type
$$K_\alpha(z)=\frac{(1-|z|^2)^{\alpha+1}}{|1-z|^{\alpha+1}}$$
which can be considered as the absolute value of the $\alpha$-harmonic Poisson kernel which was found and studied in \cite{ow} by Olofsson and Wittsten.
The related kernel $K_\alpha$ was studied further in \cite{o} by Olofsson with different methods.
The mentioned control of the $L^1$-means of $K_\alpha$ is found in Theorem~3.1 in \cite{o} and says that the function
$$M_\alpha(r)=\int_0^{2\pi} K_\alpha(re^{i\theta}) \frac{d\theta}{2\pi}$$
is bounded by $1$ for $0\leq r < 1$.
The idea behind the proof in \cite{o} is to integrate a series expansion of $K_\alpha$ and identify it as a hypergeometric function which has well-known properties.

To adapt this to our case, set $w=se^{i\theta_0}$ so that by estimating and noting that we have rotation invariance we get
$$\int_0^{2\pi} \frac{(1-r^2)^{\alpha+1}}{|1-rse^{-i(\theta-\theta_0)}|^{\alpha+2}} \frac{d\theta}{2\pi} \leq \int_0^{2\pi} \frac{(1-(rs)^2)^{\alpha+1}}{|1-rse^{i\tau}|^{\alpha+2}} \frac{d\tau}{2\pi} = \int_0^{2\pi} K_\alpha(rs e^{i\tau}) \frac{d\tau}{2\pi} \leq 1$$
for any $0\leq r < 1$.
\end{proof}

\begin{lemma}
The function
$$I_r^2(w)=\int_0^{2\pi} G(re^{i\theta},w) \frac{d\theta}{2\pi}$$
defined for $0\leq r < 1$ satisfies the pointwise limit
$$\text{$I_r^2(w)\to 0$ when $r\upto 1$ for fixed $w\in\D$}$$
and the estimate
$$I_r^2(w) \leq C_{r_0} (1-|w|^2), \quad 0<r_0\leq r < 1.$$
\end{lemma}

\begin{proof}
We proceed by explicitly calculating $I_r^2(w)$ as
\begin{equation}
\label{eq:Irw2}
I_r^2(w) = \begin{cases}
-\log|w|^2, & r\leq |w|, \\
-\log r^2, & r>|w|.
\end{cases}
\end{equation}
This follows from an easy application of the mean-value property (see Example 5.7 in \cite{st}).
Hence it is clear that for fixed $w\in\D$ we have $I_r^2(w)\to 0$ when $r\upto 1$.

To find the estimate we calculate for $0<r_0\leq r <1$:
$$\frac{I_r^2(w)}{1-|w|^2} = \begin{Bmatrix}
\displaystyle \frac{-\log|w|^2}{1-|w|^2}, & r\leq |w|\\
\displaystyle\frac{-\log r^2}{1-|w|^2},& r>|w|
\end{Bmatrix} \leq \begin{Bmatrix}
\displaystyle\frac{-\log|w|^2}{1-|w|^2}, & r\leq |w|\\
\displaystyle\frac{-\log r^2}{1-r^2},& r>|w|
\end{Bmatrix} \leq \frac{-\log r^2}{1-r^2} \leq \frac{-\log r_0^2}{1-r_0^2}$$
where the last two inequalities come from the fact that the function $\displaystyle \frac{-\log t}{1-t}$ is decreasing.
\end{proof}

These two lemmas prove the previous proposition, which in turn imply the desired boundary properties for Green's potential.

\subsection{Verifying $L_\alpha G_\alpha^\mu = \mu$}
Having found the radial $L^1(\T)$-means $I_r(w)$ above it is quite easy to show that Green's potential is integrable in $\D$:

\begin{proposition}
Green's potential $G_\alpha^\mu(z)$ belongs to $L^1(\D)$.
\end{proposition}

\begin{proof}
If we estimate using the triangle inequality and use Fubini's theorem we get
$$\|G_\alpha^\mu\|_{L^1(\D)} \leq \int_\D \left (\int_\D |G_\alpha(z,w)| dA(z) \right)  d|\mu|(w)$$
where the inner integral can be expressed in terms of the radial $L^1(\T)$-means $I_r(w)$ as
$$\int_\D |G_\alpha(z,w)| dA(z) = \int_0^1 I_r(w) 2r dr.$$
If we use the estimate (\ref{eq:L1-radial-est}) and the boundedness of $I_r^1(w)$ from Lemma~\ref{l:Iwr1} we get
$$\int_\D |G_\alpha(z,w)| dA(z) \leq C_\alpha \left( (1-|w|^2)^{\alpha+1} + (1-|w|^2)^\alpha \int_0^1 I_r^2(w) 2r dr \right).$$
The last integral here can be evaluated explicitly using (\ref{eq:Irw2}) to be
$$\int_0^1 I_r^2(w) 2r dr = -\log|w|^2 \int_0^{|w|} 2r dr + \int_{|w|}^1 2r \log r^2 dr =  1-|w|^2.$$
Hence we can conclude that
$$\int_\D |G_\alpha(z,w)| dA(z) \leq C_\alpha (1-|w|^2)^{\alpha+1} $$
which in turn implies that $\|G_\alpha^\mu\|_{L^1(\D)} < \infty$ when we integrate with respect to $|\mu|$.
\end{proof}

Knowing that $G_\alpha^\mu$ is integrable allows us to find the action of $L_\alpha G_\alpha^\mu$ as
$$\action{L_\alpha G_\alpha^\mu}{\phi} = \action{G_\alpha^\mu}{\c{L_\alpha}\phi} = \int_\D \left( \int_\D G_\alpha(z,w) d\mu(w) \right) \c{L_\alpha}\phi(z) dA(z).$$
From the proof of the previous proposition we can also deduce that $G_\alpha(z,w)\in L^1(dA\times d|\mu|)$ which justifies the use of Fubini's theorem.
So after changing the order of integration we can find the action by using Proposition~\ref{p:LG}:
\begin{align*}
\action{L_\alpha G_\alpha^\mu}{\phi} &= \int_\D \left( \int_\D G_\alpha(z,w)  \c{L_\alpha}\phi(z) dA(z) \right) d\mu(w)\\
&= \int_\D \action{L_\alpha G_\alpha(z,w)}{\phi(z)} d\mu(w) = \int_\D \phi(w) d\mu(w).
\end{align*}
Hence we conclude that $L_\alpha G_\alpha^\mu=\mu$ in the sense of distributions and we have shown that Green's potential solve Poisson's equation.

\subsection*{Acknowledgments}
I would like to thank Anders Olofsson for the useful discussions and references.

\bibliographystyle{plain}
\bibliography{green}

\end{document}